%% file: sg.tex
\documentclass[12pt]{amsart}

\usepackage{amsmath}
\usepackage{amsfonts}
\usepackage{amssymb}
\usepackage[notcite, notref]{}
\usepackage[pagebackref]{hyperref}
\include{opdef} %definicje różnych operatorów \tsum etc.

\include{envdef} % definicje środowisk: theorem, lemma, etc

\begin{document}

\date{\today}
\title{On a class of immersions between almost para-Hermitian manifolds}
\author{Piotr Dacko}
\address{}
\email{{\tt piotrdacko{\char64}yahoo.com}}

\begin{abstract} Almost para-Hermitian manifold it is manifold equipped with almost para-complex structure 
and compatible pseudo-metric of neutral signature. It is considered a class of immersions of almost 
para-Hermitian manifolds into almost para-Hermitian manifolds. Such immersions are called slant 
submanifolds. The concept  is an analogue of the idea of slant submanifold in  
 almost Hermitian geometry. There are classified pointwise slant surfaces of four dimensional almost para-Hermitian manifold. 
\end{abstract}

\maketitle

\section{Introduction}
The idea of slant submanifold arose in complex 
geometry as  generalization where  complex and totaly real submanifolds can be considered as 
limit cases. Such manifolds were extensively studied. There  are attempts to provide similar concept for para-complex, more generaly 
almost para-Hermitian manifolds and their submanifolds.  
In this  paper  the author is trying to answer the question  how to possibly
 introduce notion of {\em slant submanifold}   of almost 
para-Hermitian manifold. 

One of the goals of this work is  to show 
that there are essential differences between 
 slant submanifolds in complex geometry and what we may consider 
as they counterparts in  para-complex geometry.

\section{Preliminaries} All objects  considered in this paper are
to be smooth eg. manifold, tensor fields, etc. if not otherwise stated. All manifolds are assumed to be connected.

Let $\mathcal M$ be even dimensional manifold, $\dim\mathcal M=2n \geq 2$, equipped with pseudo-Riemannian metric $g$ of signature $(n,n)$ - such pseudo-metric will be called hyperbolic,
we also use the terms neutral pseudo-metric or Artain pseudo-metric. Customary, even if 
not correct, we drop prefix pseudo- and simply refer to $g$ as metric, if there is no chance for confusion. Similarly $(\mathcal M,g)$ is to be referred to by one of terms: hyperbolic, neutral or 
Artain manifold.  For the source of our terminology see M. Berger \cite{Ber}

Let $(\mathcal M, g)$ be Artain manifold, almost para-complex structure on $\mathcal M$,
is $(1,1)$-tensor field (affinor) $\phi$, which satisfies the conditions
\begin{equation}
\phi^2 = Id, 
\end{equation}
and for each point $q\in \mathcal M$, $\pm 1$ eigen-spaces $ \mathcal V_q^{-1}, \mathcal V_q^{+1}
\subset T_q\mathcal M$, are of the same ranks (dimensions)
\begin{equation}
\dim \mathcal V^{-1}_q = \dim \mathcal V_q^{+1}.
\end{equation}
The pair $(g,\phi)$ is called almost para-Hermitian structure if additionaly $\phi$ satisfies
\begin{equation}
g(\phi X,\phi Y) = -g(X,Y),
\end{equation}
for arbitrary vector fields. In this case it is said that almost para-complex structure is compatible and manifold with fixed almost para-Hermitian structure is called almost para-Hermitian manifold. 
On almost para-Hermitian manifold eigen-distributions, $\epsilon =\pm 1$
\begin{equation}
\mathcal V^{\epsilon} : \mathcal M\ni q \mapsto \mathcal V_q^{\epsilon},
\end{equation}
are both totally isotropic 
\begin{equation}
g(\mathcal V^{\epsilon},\mathcal V^{\epsilon})=0.
\end{equation}  

Almost para-complex structure is said to be integrable, if there is atlas of coordinate charts, such that coefficients of $\phi$ are constant, on each chart. In this case it 
is said that the structure $\phi$  is  para-complex, and manifold is called para-complex.  Almost para-Hermitian manifold with para-complex structure is called para-Hermitian.
For every almost para-Hermitian manifold
\begin{equation}
g(\phi X,Y) +g(X,\phi Y)=0,
\end{equation}
hence $\omega(X,Y)=g(\phi X,Y)$ is skew 2-form, customary called fundamental form.
%of almost para-Hermitian manifold.

According to Walker theorem on almost product structures \cite{Walk},  para-complex  manifold 
can be characterized by the following conditions
\begin{itemize}
\item[(C1)] eigen-distributions $\mathcal V^\epsilon$, $\epsilon={\pm 1}$ are completely integrable, equivalently involutive,
\item[(C2)] Nijenhuis torsion of almost para-complex structure vanishes
$$
[\phi,\phi](X,Y)=[X,Y]+[\phi X,\phi Y]- \phi ( [\phi X,Y] + [X,\phi Y])=0.
$$
\end{itemize}
Particularly on para-complex manifold there is affine, torsionless connection $\nabla$, such that $\phi$ is parallel with resp. to this connection.

  From Walker theorem it follows  that simply connected para-complex manifold 
 is diffeomorphic to Cartesian product $\mathcal M = \mathcal N_1\times \mathcal N_2$, of simply connected manifolds, 
 and eigen-distributions of $\phi$ are  kernels of canonical projections
 $\pi_i:\mathcal N_1\times \mathcal N_2\rightarrow N_i$,
 \begin{equation}
 \mathcal V^{-1} =\ker \pi_{1*},\quad \mathcal V^{+1}=\ker\pi_{2*}.
\end{equation}

Now let $(\mathcal M, \phi, g)$ be almost para-Hermitian manifold. Fix a point $q \in \mathcal M$,  linear subspace $\mathcal W_q \subset T_q\mathcal M$ 
is called singular  if equation 
\begin{equation}   
   g(v,\mathcal W_q)=0,
\end{equation}
 has non-trivial solution, and subspace $\mathcal I_q \subset \mathcal W_q$, of all such vectors is called singularity of $\mathcal W_q$.  
Note that $\mathcal I_q$, by itself is totally isotropic
 $$
 g(\mathcal I_q,\mathcal I_q) =0.
 $$
     In other words restriction $g|_{\mathcal W_q}$,
 is degenerate symmetrc form and $\mathcal I_q$ is the kernel of this form.  Similarly distribution
 \begin{equation}
 \mathcal W: \mathcal M\ni q \mapsto \mathcal W_q \subset T_q\mathcal M,
 \end{equation}
 is called singular if for some  $q$, $\mathcal W_q$ is singular.
 
Regular vector   hull   $\mathcal H_q$,  is by definition minimal non-singular space containing $\mathcal W_q$. 
Minimality means that if $\mathcal H_q \supset \mathcal H_q' \supset \mathcal W_q$, and
$\mathcal H_q'$ is also regular, then $\mathcal H_q=\mathcal H_q'$.
The hull of trivial space $\mathcal W_q=\lbrace 0 \rbrace$, is by definition trivial.
%  \begin{equation}
%  \mathcal H_q = \bigcap\; \lbrace 
%    \mathcal U_q \supset \mathcal W_q \;|\; \mathcal U_q \;{\text {is non-singular}} \rbrace.
%    \end{equation} 

Assume  $\mathcal W_q$ is singular and $\phi$-invariant, $\phi(\mathcal W_q)\subset\mathcal W_q$, $\mathcal I_q$ its singularity.
 Then $\phi(\mathcal I_q) \subset \mathcal I_q$, and  $\mathcal W_q$ admits decomposition into direct sum
 $$
 \mathcal W_q =\mathcal I_q\oplus \mathcal G_q,
 $$
 where  $\mathcal G_q$ is maximal non-singular subspace, with necessary neutral signature $(l,l)$,  and $\phi$-invariant.

It is important to notice  that  regular hull of $\phi$-invariant, singular subspace, always has neutral signature $(k+l,k+l)$, where $k =\dim \mathcal I_q$.   
In fact every regular hull $\mathcal H_q$ is isometric to orthogonal direct sum
\begin{equation}
\mathcal H_q \leftrightarrow \mathcal I_q\oplus \mathcal I_q^*\oplus \mathcal G_q,
\end{equation}•  
of space $(\mathcal G_q, g|_{\mathcal G_q})$, and the space $\mathcal I_q\oplus \mathcal I^*_q$  equipped with its canonical neutral metric.
Indeed, the  map 
$$
B:\mathcal H_q\rightarrow \mathcal H_q^*,  \quad x \mapsto \alpha = B(x), \; \alpha(y) = g(x,y),
$$
 between $\mathcal H_q$ and its dual  is linear isomorphism. We set 
$$
g^{-1}(\alpha,\beta) = g(B^{-1}\alpha,B^{-1}\beta), \quad \alpha,\beta \in \mathcal H^*_q,
$$
$B^{-1}$ denoting the inverse. Now $B$ is  isometry between  $(H_q,g)$ and $\mathcal (H_q^*,g^{-1})$. 
There is canonical  splitting 
$$
\mathcal H_q^* = \mathcal A_q\oplus B(\mathcal G_q),
$$
where $\mathcal A_q\subset \mathcal H_q^*$ is a space of all linear forms, vanishing on $\mathcal G_q$, 
and $B(\mathcal G_q)$ is image by $B$.  Note splitting is $g^{-1}$-orthogonal. To finish the proof we need only to show that $\mathcal A_q$ 
is of neutral signature.
 There is 
$$
B(\mathcal I_q)\subset \mathcal A_q,
$$ 
and any element of the quotient space $\mathcal A_q/B(\mathcal I_q)$, can be treated as linear form on $\mathcal I_q$. So 
$$
\mathcal I_q^* \subset \mathcal A_q/B(\mathcal I_q),
$$ 
and by minimality of $\mathcal H_q$, it has to be  
$$
\mathcal I_q^* = \mathcal A_q/B(\mathcal I_q).
$$
 As $B(\mathcal I_q)$
is totally isotropic in $\mathcal A_q \cong \mathcal I_q\oplus \mathcal I_q^*$, $\mathcal A_q$ has neutral signature $(k,k)$, where $k=\dim \mathcal I_q$. 
For the concepts of regular hull see also \cite{Ber}, ch.7.
 
Now let $\mathcal W_q\subsetneq T_q\mathcal M$ be a non-singular, we set 
\begin{equation}
\mathcal K_q = \mathcal W_q^\perp\cap(\phi \mathcal W_q)^\perp,
\end{equation} 
$\phi \mathcal W_q$ denotes the image by $\phi$.
Directly we find  $\phi(\mathcal K_q)=\mathcal K_q$. 

%B the above
%proposition
%$$
%\mathcal I_q^\epsilon \oplus \mathcal K_q^\epsilon\subset 
%\mathcal V_q^\epsilon \cap \mathcal W_q^\perp,
%\quad and\quad
%\mathcal K_q^{-\epsilon} \subset \mathcal V_q^{-\epsilon}\cap \mathcal W_q^\perp,
%$$
%but from definition of $\mathcal K_q$ we obtain that arbitrary eigen-vector of $\phi$ 
%has to belong to $\mathcal K_q$, hence
%$$
%\mathcal I_q^\epsilon \oplus \mathcal K_q^\epsilon = 
%\mathcal V^\epsilon \cap \mathcal W_q^\perp, \quad
%\mathcal K_q^{-\epsilon} = \mathcal V_q^{-\epsilon}\cap \mathcal W_q^\perp.
%$$  
%We say that $\mathcal K_q$ is proper if its vector hull is a proper subspace of 
%$\mathcal W_q^\perp$. 

\section{Slant immersions}
Let $(\mathcal M, \phi,g)$ and $(\mathcal S, \tilde g,\tilde \phi)$ be almost para-Hermitian manifolds, $f:\mathcal S\rightarrow \mathcal M$ - isometric immersion $\tilde g=f^*g$. 
We refer to $(g,\phi)$, $(\tilde g,\tilde\phi)$ as outer and inner almost para-Hermitian structures,
resp. 
 
For $p\in \mathcal S$, $q=f(p)\in f(\mathcal S)\subset \mathcal M$, we denote by $T_qf \subset T_q\mathcal M$, subspace tangent to image $f(\mathcal S)$ at $q$, and by
$N_qf\subset T_q\mathcal M$  subspace normal to $f(S)$,  hence
we have orthogonal decomposition 
\begin{equation}
T_q\mathcal M =  T_qf\oplus N_qf.
\end{equation}

For a vector $v\in T_p\mathcal S$, 
$$
(\phi f_*(v))^\top,\quad(\phi f_*(v))^\perp,
$$ 
denote orthogonal
projections of $\phi f_*(v)$ onto tangent and normal spaces.  The correspondence 
\begin{equation}
 f_*(v) \mapsto (\phi f_*(v))^\perp
\end{equation}
defines linear endomorphism $A:T_qf\rightarrow N_qf$, called normal map. 

  Note that  $w \in T_qf$ belongs to the kernel $\ker A$ $\iff$ $\omega(w,N_qf)=0$, hence 
   \begin{equation}
 \ker A = T_qf\cap(\phi N_qf)^\perp = N_qf^\perp\cap (\phi N_qf)^\perp,
 \end{equation}
 in particular $\phi(\ker A)=\ker A$.  
 
  Immersion $f$ is called slant at the point $p\in \mathcal S$, if 
there is  orthogonal decomposition  
 \begin{equation}
 T_qf =  \mathcal H_q\oplus \mathcal H_q^\perp,
 \end{equation}
where $\mathcal H_q$, is regular vector hull of the kernel of normal map $A$, there is
$\lambda = \lambda_q \in \mathbb R$, and following conditions are satisfied
\begin{eqnarray}
(\phi f_*(v))^\top & =& \epsilon f_*(\tilde \phi v), \quad f_*(v)\in \mathcal H_q, \\[+4pt]
(\phi f_*(w))^\top & = &\lambda f_*(\tilde \phi w), \quad f_*(w) \in \mathcal H_q^\perp ,
\end{eqnarray}
$v\in T_p\mathcal S$. The case $\epsilon =+1$ we call semi-invariant, and $\epsilon=-1$, semi-anti-invariant. It is allowed that one of $\mathcal H_q$ or
$\mathcal H_q^\perp$ is trivial. 
Having above pointwise definition we will said    immersion $f$ is slant if it  is  slant  at every point of $\mathcal S$. 
 
By the previous section both $\mathcal H_q$ and $\mathcal H_q^\perp$ are Artain subspaces in $T_qf$, now if we go back to the inner structure 
on $\mathcal S$, we see that there is corresponding orthogonal splitting  $T_p\mathcal S = \mathcal Q_p\oplus \mathcal Q_p^\perp$, 
$\mathcal H_q= f_*(\mathcal Q_p)$, $\mathcal H_q^\perp= f_*(\mathcal Q_p^\perp)$,  
into Artain subspaces, and each of these subspaces are invariant  
$\tilde \phi (\mathcal Q_p)=\mathcal Q_p$, $\tilde \phi (\mathcal Q_p^\perp) = \mathcal Q_p^\perp$.   

In the  definition above there is no requirement concerning uniquenes of such decomposition.
 The proposition below asserts that simply such requirement is unneccesary provided 
$\epsilon \neq \lambda$. 

\begin{proposition} Assume immersion is slant at $p\in \mathcal S$,  $q=f(p)$, let 
\begin{gather*}
T_qf = \mathcal H_q\oplus \mathcal H_q^{\perp}, \\
T_qf = \mathcal H_q'\oplus \mathcal H_q^{\prime \perp},
\end{gather*}
be orthogonal splittings corresponding to $(\epsilon,\lambda)$ and $(\epsilon,\lambda')$, resp. If $\epsilon \neq \lambda$ then 
\begin{equation*}
\mathcal H_q =\mathcal H_q'\quad \mathcal H_q^\perp = \mathcal H_q^{\prime \perp}.
\end{equation*}
\end{proposition}
\begin{proof}
Let 
\begin{gather*}
\mathcal H_q = (\mathcal I_q\oplus \mathcal I_q')\oplus \mathcal G_q, \\
\mathcal H_q' = (\mathcal I_q \oplus \mathcal I_q'')\oplus \mathcal G_q, 
\end{gather*}
where $\ker A = \mathcal H_q\cap \mathcal H_q' = \mathcal I_q\oplus\mathcal G_q$.
 We may assume $\mathcal I_q''\neq \{0\}$. 
Let $f_*(v)\in \mathcal I_q''$, and
$$
f_*(v) = f_*(u_1)+f_*(u_2), \quad f_*(u_1) \in \mathcal H_q, \quad f_*(u_2)\in \mathcal H_q^{\perp},
$$
then 
\begin{eqnarray*}
\phi f_*(v) &=& \epsilon f_*(\tilde \phi v) + Af_*(v) = \epsilon f_*(\tilde \phi u_1)+\epsilon f_*(\tilde \phi u_2)+Af_*(v), \\
\phi f_*(v) &=& \phi f_*(u_1)+\phi f_*(u_2) = \epsilon f_*(\tilde \phi u_1)+\lambda f_*(\tilde\phi u_2) +Af_*(v),
\end{eqnarray*}
hence $\epsilon f_*(\tilde \phi u_2) = \lambda f_*(\tilde \phi u_2)$, assumption   $\epsilon \neq \lambda$ follows $u_2=0$, and 
$\mathcal I_q'' \subset \mathcal H_q$. 
\end{proof}

\section{Slant surfaces}
In this section we study slant surfaces in four dimensional almost para-Hermitian manifold.
Let $\mathcal S$ be para-Hermitian  surface, that is real 2-dimensional
manifold with almost para-Hermitian structure $(\tilde \phi,\tilde g)$,
and $\mathcal M$ be 4-dimensional almost para-Hermitian manifold with
almost para-Hermitian structure $(\phi, g)$, $f$ - isometric  immersion
$$
f: (\mathcal S, \tilde\phi, \tilde g)\rightarrow (\mathcal M, \phi,g).
$$ 
According to previous section there are following possibilities:
\begin{itemize}
\item[a)] the kernel of normal map is non-trivial, in this case if $A\neq 0$, then the kernel is exactly one-dimensional,
\item[b)] normal map is trivial $A=0$ wich means surface is almost para-Hermitian submanifold,
\item[c)] the kernel of $A$ is trivial and $A \neq 0$, then $A$ is isomorphism between tangent and normal spaces, 
in this case totaly real surfaces are treatead as limit case.  
\end{itemize}•
For surfaces assumption   $f$ is slant is euivalent to the following identity, which covers all above possibilities
\begin{equation}
\label{phif}
\phi f_*(v) =\lambda f_*(\tilde\phi v)+Af_*(v), \;\; \lambda \in \mathbb R. 
\end{equation}
Then
\begin{eqnarray}
\label{gAA}
     & g_q(Af_*(v),Af_*(w)) = (\lambda^2-1)\tilde g_p(v,w), &\\
\label{phiA}  &  \phi_q Af_*(\tilde \phi v) = (1-\lambda^2)f_*(\tilde\phi v)-\lambda Af_*(v),  &
\end{eqnarray}
hence the composition map $Af_*: T_p\mathcal S\rightarrow N_{f(p)}f$, for $|\lambda|\neq 1$, is conformal.
 
The goal of next proposition is to describe completely coefficients of $\phi$,  that is coefficients of almost para-complex structure 
of {\em ambient} space, at point where immersion is slant.
 We set ${\rm ker}\, A$, ${\rm im}\, A$, as kernel and image of $A$.
 Let fix orthonormal frame 
$(v_1^+,v_2^-=\tilde\phi v_1^+)$ of $T_p\mathcal S$. Vectors $(e_1^+=f_*(v_1),e_2^-=f_*(v_2))$
are orthonormal and span tangent plane 
$T_qf$.
\begin{proposition}
\label{locstr}
There exists orthonormal frame $(e^+_3, e^-_4) $ of normal plane, such that  
 
    \begin{equation}
      \label{phi1}%
      \begin{array}{ll}
        \phi e_1 =\lambda e_2 +c_\lambda e_3, & \quad\phi e_2 = \lambda e_1+c_\lambda e_4, \\[+4pt]
        \phi e_3 = -c_\lambda e_1-\lambda e_4, & \quad\phi e_4 = -c_\lambda e_2-\lambda e_3, \\[+4pt]
        |\lambda|>1,\;\; c_\lambda=\sqrt{\lambda^2-1},     
      \end{array}% 
    \end{equation}
  \begin{equation}
      \label{phi2}
      \begin{array}{ll}
        \phi e_1 = \lambda e_2 +c_\lambda e_4, & \quad\phi e_2 = \lambda e_1+c_\lambda e_3, \\[+4pt]
        \phi e_3 = c_\lambda e_2 - \lambda e_4, & \quad\phi e_4
        = c_\lambda e_1-\lambda e_3, \\[+4pt]
      |\lambda| <1,\;\;  c_\lambda= \sqrt{1-\lambda^2},  
      \end{array}
    \end{equation}
  
\begin{equation}
\label{phi3}
\begin{array}{ll} 
 \phi e_1 = \lambda e_2 +a_0(e_3+e_4), &  \quad
\phi e_2 = \lambda e_1 -\frac{\epsilon}{\lambda} a_0 (e_3+e_4), \\[+4pt]
\phi e_3 = \epsilon e_4 -a_0(e_1+\frac{\epsilon}{\lambda} e_2), & \quad
\phi e_4 = \epsilon e_3 +a_0(e_1+\frac{\epsilon}{\lambda} e_2),\\[+4pt]
|\lambda| =1,\;\;  |\epsilon|= 1, \;\;a_0 \in \mathbb R,
\end{array}
\end{equation}

\end{proposition}
\begin{proof}
We see, that  $|\lambda|\neq 1$, follows ${\rm ker}\,A =0$.
We set  $e_3$, $e_4$ as
$$
e_3 = Af_*(v_1)/c_\lambda, \;\; e_4=Af_*(v_2)/c_\lambda, \;\; c_\lambda=\sqrt{\lambda^2-1} > 0,
$$  
for $|\lambda| > 1$, and
$$
e_3=Af_*(v_2)/c_\lambda,\;\; e_4 = Af_*(v_1)/c_\lambda,\;\; c_\lambda=\sqrt{1-\lambda^2} > 0,
$$
for $|\lambda|<1$.  The above formulas 
are not same: note interchange in $v_1$ and $v_2$.
Then $(e_3,e_4)$ are orthonormal and $g(e_3,e_3)=+1$, $g(e_4,e_4)=-1$.  
We set  $v=v_1$, $v=v_2$, in (\ref{phif}), by definition of $e_3$, $e_4$, we find expressions for $\phi e_1$, $\phi e_2$.   
Then expressions  $\phi e_3$, $\phi e_4$,  come  from (\ref{phiA}).

To prove c), if $a=0$, then $A=0$, and $f$ is invariant or anti-invariant, orthogonal plane $N_qf$ is $\phi$-invariant, for given $\epsilon = \pm 1$, we can always find orthonormal base $(e_3^+,e_4^-)$ of $N_qf$, such that $\phi e_3 = \epsilon e_4$, $\phi e_4=\epsilon e_3$. So let $a\neq 0$, then by (\ref{gAA}), 
$l={\rm im}\, A$,  is isotropic line in normal plane, hence
$$
Af_*(v)=\alpha(v)n, \;\; v\in T_p\mathcal S,
$$ 
for 1-form $\alpha$, and  vector n, which spans  $l$. 
By (\ref{phiA}) 
 \begin{equation}
 \label{phin}
\phi n = \epsilon n, \;\; \alpha(\tilde \phi v) = -\frac{\lambda}{\epsilon}\alpha(v), \;\;\lambda =\pm 1.
\end{equation}
We may assume that $n=e_3+e_4$ for  some orthonormal base $(e^+,e^-)$ of $N_qf$. 
 Now 
\begin{equation}
\label{phivc}
\phi f_*(v) = \lambda f_*(\tilde \phi v)+\alpha(v)(e_3+e_4),
\end{equation}
setting $v=v_1$, and  $v=v_2$,  we  find coefficients for $\phi e_1$ and $\phi e_2$,
then by anti-symmetry  $g(\phi e_i,e_j)= -g(e_i,\phi e_j)$ --  coefficients for $\phi e_3$, $\phi e_4$, where  $a_0=\alpha(v_1)$, finally we verify that $\phi^2 e_i = e_i$, 
$i=1,\ldots,4$.  
\end{proof}
Let for simplicity assume that  $f: \mathcal S \rightarrow \mathcal M$, is 
slant embedding, so $f$ is slant at every point. Because conformal change of 
metric in terms of orthonormal frame is expressed by multiplying each frame 
element by the same constant, the above  proposition tells, that if we 
change conformally metrics on manifolds $\mathcal S$, and $\mathcal M$, 
without violating the condition, that $f$ is isometric immersion,  then $f$ 
still became slant, with exactly the same slant factors.
\begin{corollary}Isometric embedding 
 of para-Hermitian surface into four 
dimensional almost para-Hermitian manifold, is slant in its immersion 
conformal class
$\mathcal C(f)$. 
\end{corollary}
For example, by conformal invariance, we can replace ambient space of constant sectional curvature, by locally flat space, to find any slant surface immersion into  manifold of constant sectional curvature. From other hand, as 
every para-Hermitian surface is locally conformally flat, question of 
existence of slant immersion of given surface, can be, at least locally, reduced to 
the same question for locally flat surface.

At the end of this section we discuss relation between slant immersions 
and Lagrangian surfaces. Assuming $\mathcal M$ is almost para-K\"ahler, 
fundamental form $\omega$ is closed, and it is symplectic form on $\mathcal M$, 
so in natural manner any almost para-K\"ahler manifold is symplectic manifold. 
Surface $f:\mathcal S \rightarrow \mathcal M$, is called Lagrangian, if
$f^*\omega=0$. Of course a priori Lagrangian surface does not carry almost 
para-Hermitian structure. From other hand totally real immersion 
of para-Hermitian surface $f:\mathcal S \rightarrow \mathcal M$ is 
Lagrangian submanifold, we verify this directly.  In the view of the Proposition \ref{locstr}, totally real immersion can be treated as limit, where 
slant factor tends to zero, $\lambda \rightarrow 0$. Such point of view is 
natural when considering deformations of slant immersions. 
Natural question arises: can be given slant immersion deformed into totally 
real immersion? Thus as limit we would obtain 
(for $\mathcal M$ almost para-K\"ahler) Lagrangian submanifold. For 
example in \cite{NG}, there are classified Lagrangian immersions 
$f:\mathcal S \rightarrow \Sigma_1\times\Sigma_2$, into 
products  of para-K\"ahler surfaces. In this case the above question reads: are 
there almost para-Hermitian structure on $\mathcal S$, and family of 
slant immersions $f_t$, such that $f_{t_0} = f$?

In the view of the Proposition \ref{locstr}. 
There are frames 
of vectors field defined only along given slant surface. However if we able to 
extend smoothly these frames onto open neighborhood of such surface, in  the 
manner that  all formulas from the Proposition \ref{locstr}. are still valid, 
then we come directly to  described below examples. 

\noindent {\bf Example 1.} Let $\mathcal G$ be a Lie group, let $(e_1,\ldots,e_4)$, be a basis
 of its Lie algebra $\mathfrak g$, with 
 commutators
 \begin{gather*}
 [e_1,e_2] = [e_3,e_4]=0, \quad [e_1,e_3]=e_3, \quad [e_1,e_4]=e_4, \\ 
 [e_2,e_3] = -e_4,  \quad [e_2,e_4] = e_3.
 \end{gather*}
We define almost para-Hermitian structure $(\phi_\lambda,g)$ on $\mathcal G$ as follows:
 $(e_1^+,e_2^-,e_3^+,e_4^-)$ is $g$-orthonormal frame, and almost para-complex structure $\phi_\lambda$ is defined 
by  set of equations as  (\ref{phi1}), (\ref{phi2}) or (\ref{phi3}). Distribution $\mathcal D$, 
spanned by vector fields $e_1,e_2$, is completely integrable. Let $\mathcal S$ be a leaf of corresponding 
foliation, and $\iota: \mathcal S \rightarrow \mathcal G$, inclusion map. We equip $\mathcal S$ with almost
para-Hermitian structure $(\tilde \phi, \tilde g)$, 
\begin{equation*}
\iota_*\tilde \phi=\dfrac{1}{\lambda}\phi|_{\mathcal S}\,\iota_* , \quad \tilde g = \iota^* g,
\end{equation*}
where $\phi|_{\mathcal S}$ denotes restriction of ambient paracomplex structure to  leaf.
Now inclusion $\iota:\mathcal S\rightarrow G$, is slant immersion of para-Hermitian surface $(\mathcal S,\tilde \phi, \tilde g)$ into 
almost para-Hermitian manifold $(\mathcal G,\phi_\lambda, g)$, with constant slant factor $=\lambda$.
\vspace{0.25 cm}

\noindent {\bf Example 2.} Let $\mathcal M = (\mathbb R^4,g)$, where $g$ is neutral flat pseudo-metric. By 
$(x^1,\ldots,x^4)$ we denote global coordinates on $\mathbb R^4$, 
and $$
e_1^+=\frac{\partial}{\partial x^1},\quad e_2^-=\frac{\partial}{\partial x^2},
\quad e_3^+=\frac{\partial}{\partial x^3},\quad
 e_4^-=\frac{\partial}{\partial x^4},
$$ is corresponding 
global orthonormal frame of vector fields. For  smooth function 
$\lambda:\mathbb R^4\rightarrow \mathbb R $, we define almost paracomplex 
structure $\phi_\lambda$ on $\mathcal M$ as in  (\ref{phi1}) or (\ref{phi2}). 
Clearly there have to be satisfied conditions $|\lambda| > 1$ or 
$|\lambda| < 1$. Let $\iota: \mathcal S = \mathbb R^2 \subset \mathcal M$, 
be a plane given  by $x^3=const.$, $x^4=const.$, we define 
almost para-Hermitian structure $(\tilde\phi,\tilde g)$ on 
$\mathcal S'=\mathcal S \setminus \lbrace \lambda=0 \rbrace $,  
$$
\iota_*\tilde \phi = \dfrac{1}{\lambda}\phi|_{\mathcal S'}\,\iota_*, \quad
\tilde g = \iota^*g,
$$   
we additonally assume that $\lambda|_{\mathcal S'}$ is non-constant. With this structure 
inclusion $\iota:\mathcal S'\rightarrow \mathcal M$ is slant immersion 
with non-constant slant factor $\lambda|_{\mathcal S'}$.

\bibliographystyle{amsplain}
 
\end{document}

%% file: opdef.tex
\def\ker{\mathop{\rm ker}\nolimits}

\def\dfrac{\displaystyle \frac}

%% file: envdef.tex
\newcounter{prop}
\newtheorem{proposition}[prop]{Proposition}
\newcounter{coro}
\newtheorem{corollary}[coro]{Corollary}